\documentclass[a4paper,12pt]{article}
\usepackage{amsthm}
\usepackage{amsmath}
\usepackage{amssymb}
\usepackage{latexsym}
\usepackage{url}
\usepackage{color}
\usepackage{comment}
\numberwithin{equation}{section}

\def\R{{\bf R}}

\def\N{{\bf N}}

\def\d{\displaystyle}
\def\e{{\varepsilon}}

\def\wt{\widetilde}

\def\p{\partial}

\newcommand{\dt}{\partial_{t}}%

\newcommand{\al}{\alpha}
\newcommand{\be}{\beta}
\newcommand{\ga}{\gamma}

\newcommand{\ka}{\kappa}
\newcommand{\si}{\sigma}
\newcommand{\la}{\lambda}

\renewcommand{\th}{\theta}
\newcommand{\de}{\delta}
\newcommand{\om}{\omega}

\newcommand{\na}{\nabla}

\newcommand{\I}{\infty}

\newcommand{\EQS}[1]{\begin{align} #1 \end{align}}
\newcommand{\EQQS}[1]{\begin{align*} #1 \end{align*}}
\newcommand{\EQQ}[1]{\begin{equation*} \begin{split} #1
 \end{split} \end{equation*}}

\newcommand{\LR}[1]{{\langle {#1} \rangle }}

\newtheorem{thm}{Theorem}[section]

\newtheorem{lem}{Lemma}[section]
\newtheorem{prop}{Proposition}[section]
\newtheorem{rem}{Remark}[section]
\newtheorem{definition}{Definition}[section]

\title{On the critical decay for the wave equation with a cubic convolution in 3D}
\author{
Tomoyuki Tanaka
\footnote{Graduate School of Mathematics, Nagoya University, Chikusa-ku, Nagoya, 464-8602, Japan. e-mail: d18003s@math.nagoya-u.ac.jp}
\ and \
Kyouhei Wakasa
\footnote{Department of Creative Engineering, National Institute of Technology, Kushiro College, 2-32-1 Otanoshike-Nishi, Kushiro-Shi, Hokkaido 084-0916, Japan. e-mail: wakasa@kushiro-ct.ac.jp.
}
}
\date{
\[
\begin{array}{llll}
\mbox{\footnotesize{\bf Keywords:}}
& \mbox{\footnotesize Wave equation, Cubic convolution, Global existence,}\\
& \mbox{\footnotesize Blow-up, Lifespan, Critical exponent}\\
\mbox{\footnotesize{\bf MSC2010:}}
& \mbox{\footnotesize Primary 35B33 Secondary 35B44, 35L05, 35L71, 35B45}\\
\end{array}
\]
}
\begin{document}
\maketitle
\begin{abstract}
We consider the wave equation with a cubic convolution
$$\dt^2 u-\Delta u=(|x|^{-\ga}*u^2)u$$
in three space dimensions.
Here, $0<\ga<3$ and $*$ stands for the convolution in the space variables.
It is well known that if initial data are smooth, small and compactly supported, then $\ga\ge2$ assures unique global existence of solutions.
On the other hand, it is also well known that solutions blow up in finite time for initial data whose decay rate is not rapid enough even when $2\le \ga<3$.
In this paper, we consider the Cauchy problem for $2\le \ga<3$ in the space-time weighted $L^\I$ space in which functions have critical decay rate.
When $\ga=2$, we give an optimal estimate of the lifespan.
This gives an affirmative answer to the Kubo conjecture (see Remark right after Theorem 2.1 in \cite{K04}).
When $2<\ga<3$, we also prove unique global existence of solutions for small data.
\end{abstract}


\section{Introduction}

\quad \ \ We consider the following Cauchy problem:
\EQS{\label{IVP}
\begin{cases}
  \dt^2 u-\Delta u=(V_\ga*u^2)u, &(x,t)\in\R^3\times [0,T),\\
  u(x,0)= u_0(x), &x\in\R^3,\\
  \dt u(x,0)= u_1(x), &x\in\R^3.
\end{cases}
}
Here, $V_\ga(x)=|x|^{-\ga}$ for $0<\ga<3$ and $*$ stands for the convolution in the space variables.
For the initial data, we assume $(u_0,u_1)\in C^1(\R^3)\times C(\R^3)$.
The stationary problem for the equation with $\gamma=1$ and a mass term is a model for the Helium atom which is proposed by Hartree. Also, Menzala and Strauss \cite{MS82} studied the Cauchy problem  (\ref{IVP}) with a mass term.

The Cauchy problem for the wave equation with power nonlinearity $|u|^p$, which reads
\EQS{\label{IVP2}
\begin{cases}
  \dt^2 u-\Delta u=|u|^p, &(x,t)\in\R^n\times [0,T),\\
  u(x,0)= u_0(x), &x\in\R^n,\\
  \dt u(x,0)= u_1(x), &x\in\R^n,
\end{cases}
}
where $p>1$, $n\ge2$, and $u_0,u_1\in C^{\infty}(\R^n)$, has been extensively studied by many authors.
We review known results for the case that
$(u_0,u_1)$ has a compact support.
When $n\ge2$, it is known that the following Strauss' conjecture holds.
That is, there exists a critical exponent $p_0(n)$ such that the solution of (\ref{IVP2}) exists globally in time for small data if $p>p_0(n)$,
and the solution of (\ref{IVP2}) blows up in finite time for positive initial data if $1<p\le p_0(n)$.
Here, $p_0(n)$ is a positive root of the quadratic equation $(n-1)p^2-(n+1)p-2=0$. This was first
showed by John \cite{J79} except for $p=p_0(3)(=1+\sqrt{2})$ in $n=3$.
See \cite{GLS97,G81a,G81b,R87,S85,Si84,YZ06,Z07} for contributions to this conjecture, \cite{LZ14,L90,LS96,Ta15,TW11,Z92_three,Z93,ZH14} for the estimate of the lifespan and \cite{AT92,Asa86,K96,KK98,K93,T95,TUW10,T92,T93,T94} for results to Strauss' conjecture for slowly decaying data.

We turn back to our original problem (\ref{IVP}), and we recall deeply related results in three space dimensions $n=3$.
Hidano \cite{H20} proved the small data scattering to \eqref{IVP} for $2<\gamma<5/2$. On the other hand, he also proved the small data blow-up result to (\ref{IVP}) with $0<\gamma<2$ for some positive initial data with a compact support.
From this result, we can see that $\ga=2$ is the critical exponent to \eqref{IVP}.
When $\gamma=2$, Kubo \cite{K04} showed the unique global existence of the small solution for slowly decaying data.
More precisely, he considered \eqref{IVP} with initial data $(u_0,u_1)\in C^1(\R^3)\times C(\R^3)$ which is small in the following norm:
\EQS{\label{ge_condi3}
\sup_{x\in\R^3}
  \{(1+|x|)^{\ka}|u_0(x)|+(1+|x|)^{\ka+1}(|\na u_0(x)|+|u_1(x)|)\}
}
for $3/2<\ka<2$.
This is remarkably different from the wave equation with a power nonlinearity $|u|^p$ since the solution of \eqref{IVP2} blows up in finite time when $p=p_0(3)$.

In Remark right after Theorem 2.1 in \cite{K04}, Kubo also conjectured that the solution blows up in finite time when $\ka=3/2$.
This conjecture for $\ga=2$ seems to be natural from the scaling argument.
If $u$ is a solution to \eqref{IVP} with initial data $(u_0(x),u_1(x))$ on $[0,T]$, then
$$u_\si(x,t)=\si^{(5-\ga)/2}u(\si x,\si t)$$
is also a solution to \eqref{IVP} with initial data $(\si^{(5-\ga)/2}u_0(\si x),\si^{(7-\ga)/2}u_1(\si x))$
on $[0,T/\si]$.
We note that
\EQQS{
  \|(\si^{(5-\ga)/2}u_0(\si x),\si^{(7-\ga)/2}u_1(\si x))\|_{\tilde{Y}(\ka)}
  =\si^{(5-\ga)/2-\ka}\|(u_0(x),u_1(x))\|_{\tilde{Y}(\ka)},
}
where
\EQQS{
  \|(u_0,u_1)\|_{\tilde{Y}(\ka)}
  =\sup_{x\in\R^3}
    \{|x|^{\ka}|u_0(x)|+|x|^{\ka+1}(|\na u_0(x)|+|u_1(x)|)\}
}
(which is a homogeneous version of \eqref{ge_condi3}).
So, the scale transformation preserves the above norm when $\ka=(5-\ga)/2$.
This observation gives us an intuition that $\ka=(5-\ga)/2$ can be a threshold to devide global existence and blow up result.
In fact, when $2<\ga<3$, Tsutaya \cite{T03} gave a negative answer to the Kubo conjecture.
He studied (\ref{IVP}) for data $(u_0,u_1)\in C^1(\R^3)\times C(\R^3)$ which is small
\footnote{To be precise, initial data in \cite{T03} are more regular than $C^1(\R^3)\times C(\R^3)$.}
in \eqref{ge_condi3} and proved the solution exists globally in time for $2<\ga<3$ and $(5-\ga)/2<\ka<2$.
On the other hand, for $1/2<\kappa<(5-\gamma)/2$ and data satisfying
\begin{equation}
\label{blow-up-asm}
u_0(x)\equiv0 \quad \mbox{and}\quad u_1(x)\ge \frac{B}{(1+|x|)^{\ka+1}}\quad \mbox{for}\quad |x|\ge1
\end{equation}
for $B>0$,
he proved the blow-up result to the problem (\ref{IVP}) with $0<\gamma<3$.
In particular, $(5-\ga)/2<3/2$ holds when $2<\ga<3$, so Tsutaya proved global existence of the solution for $\ka=3/2$.
Therefore, the Kubo conjecture is not true when $2<\ga<3$ and remains open for the critical case $\ga=2$.
Our plan is to give an answer to the Kubo conjecture for $\ga=2$ and $\ka=3/2$, and we also treat the Cauchy problem \eqref{IVP} for $2<\ga<3$ and $\ka=(5-\ga)/2$.
For $\ga=2$, we prove the lower and upper bound of the lifespan.


Before stating our main results, we give the definition for the solution and the lifespan.

\begin{definition}[Solution, Lifespan]
\label{def1}
\begin{itemize}
\vskip5pt
\item (Solution): Let $T>0$ and $(u_0,u_1)\in C^{1}(\R^3)\times C(\R^3)$. We say that the function
$u$ is a solution to the Cauchy problem \eqref{IVP} if $u$ belongs to the class $C(\R^3\times [0,T))$ and satisfies the integral equation (\ref{IE_u}).
\item (Lifespan): We call the maximal existence time $T=T(\e u_0,\e u_1)$ to be lifespan. For initial data $(\e u_0,\e u_1)$, the lifespan $T=T(\e u_0,\e u_1)$ is denoted by $T(\e)$, namely
\begin{align*}
    T(\varepsilon)
    &:=\sup\big\{T\in (0,\infty]:\\
    &\text{there exists a unique solution $u$ to \eqref{IVP} with $(\e u_0,\e u_1)$ on $\R^n\times [0,T)$}\big\}.
\end{align*}
\end{itemize}
\end{definition}

The first two results are devoted to the lifespan of the solution when $\ga=2$ and $\ka=3/2$.
Especially, Theorem \ref{thm_bu} gives an affirmative answer to the Kubo conjecture.

\begin{thm}[Lower bound of the lifespan]\label{thm_lb}
  Let $\ga=2$ and $(u_0,u_1)\in Y(3/2)$.
  Then, there exist $C>0$ and $\e_0=\e_0(u_0,u_1)>0$ such that the lifespan $T(\e)$ of solutions \eqref{IVP} with $(u(x,0),\dt u(x,0))=(\e u_0(x),\e u_1(x))$ satisfies
  $$T(\e)\ge \exp(C\e^{-2})$$
  for $\e\in(0,\e_0]$.
\end{thm}

\begin{rem}
  A class of initial data $Y(\ka)$ is defined in Section 2.
\end{rem}


\begin{thm}[Upper bound of the lifespan]\label{thm_bu}
  Let $\ga=2$. Assume that (\ref{blow-up-asm}) with $\ka=3/2$ holds.
  Then, there exist $C>0$ and $\e_0=\e_0(u_1)>0$ such that the lifespan $T(\e)$ of solutions \eqref{IVP} with $(u(x,0),\dt u(x,0))=(0,\e u_1(x))$ satisfies
  $$T(\e)\le \exp(C\e^{-2})$$
  for $\e\in(0,\e_0]$.
\end{thm}

On the other hand, we can show the existence of global solutions to \eqref{IVP} with small data when $2<\ga<3$ and $\ka=(5-\ga)/2$.

\begin{thm}[Global existence for critical decaying data]\label{thm_ge}
  Let $2<\ga<3$ and $(u_0,u_1)\in Y((5-\ga)/2)$.
  Then, there exists $\e_0=\e_0(\ga,u_0,u_1)>0$ such that for any $\e\in(0,\e_0]$ there exists a unique solution $u\in X_\ga(\I)$ of \eqref{IVP} with $(u(x,0),\dt u(x,0))=(\e u_0(x),\e u_1(x))$.
\end{thm}

\begin{rem}
  The solution space $X_\ga(T)$ is defined in Section 2.
\end{rem}

Now, we state ideas to obtain our results.
The key ingredient to show Theorem \ref{thm_lb} is Proposition \ref{prop_potential}, which is a refinement of Proposition 2.2 in \cite{T14} in some sense.
Simply following the idea of Proposition 2.2 in \cite{T14}, we get only the following
\EQS{\label{eq_tsutaya}
  |(V_2 * u^2)(x,t)|
  \lesssim \frac{1+(\log(1+t+|x|))^2}{(1+t+|x|)^2}\|u\|_X^2,
  \quad (x,t)\in\R^3\times[0,T)
}
(see \eqref{sols_norm} for the definition of $\|\cdot\|_X$).
We have a square of $\log(1+t+|x|)$ in the right hand side since we are in the doubly critical situation $\ga=2$ and $\ka=3/2$.
Then, a standard argument (see the proof of Theorem \ref{thm_lb}) implies that $T(\e)\ge\exp(\tilde{C}\e^{-1})$ for some $\tilde{C}>0$ and small $\e>0$.
This lower bound is insufficient, and we need to eliminate one of the logarithmic growth in \eqref{eq_tsutaya}.
However, it seems to be difficult to do so without losing any powers of $(1+t+|x|)^{-2}$.
This is the price to pay for the estimate
in Proposition \ref{prop_potential} (see e.g. (ii) of Case 1 in the proof of Proposition \ref{prop_potential} for the worst estimate).
One can find a similar argument in the proof of Theorem 2.1 in \cite{K04}.
The most important point in the proof is that the price does not exceed $1/2$, i.e, we need a estimate
\EQS{\label{eq_conv1}
  |(V_2 * u^2)(x,t)|
  \lesssim
  \frac{1+\log(1+t+|x|)}{(1+t+|x|)^{a}(1+|t-|x||)^{b}}\|u\|_X^2,
  \quad (x,t)\in\R^3\times[0,T)
}
for $a+b=2$ and $0<b<1/2$.
For that purpose, we use (a) of Lemma \ref{lem_2.1} with a different choice of $\de$ depending on where $(x,t)$ is located.
If $(x,t)$ is away from the light cone (i.e., $t\ge 2|x|$ or $2t\le |x|$), we employ (a) of Lemma \ref{lem_2.1} with $\de=1$.
On the one hand, we use (a) of Lemma \ref{lem_2.1} with $\de=1/4$ (then $b=1/4$) when $(x,t)$ is close to the light cone (i.e., $|x|/2\le t\le 2|x|$).
To be precise, it suffices to choose $0<\de<1/2$ in this case.
In other words, if we choose $\de=1/2$ (which corresponds to Tsutaya's argument in \cite{T14}), then we have $b=1/2$, and we may have another logarithmic growth in the estimate of the Duhamel term.
If we choose $\de>1/2$, then $b>1/2$, and it is difficult to close the estimate (in the sense of Proposition \ref{prop_duhamel}) because of the lack of the power of $(1+t+|x|)$.

Proposition \ref{prop_potential2} is the key estimate for the proof of Theorem \ref{thm_ge}.
Indeed, by the same spirit as stated above,
we lose a small order of the decay $(1+t+|x|)^{-2}$ to compensate the logarithmic growth which comes from the critical decay rate of initial data, which results in we have the logarithmic growth free estimate.

To get the upper bounds of the lifespan, we use the iteration argument by \cite{J79}, together
with the slicing method in Agemi, Kurokawa and Takamura \cite{AKT00}.
The slicing method is an useful method for obtaining the logarithmic growth of the solution in the critical nonlinearity. An essential part in this method is to slice the integral domain after applying the integration.

The key fact in Theorem \ref{thm_bu} is to get a logarithmic growth for the convolution term:
\begin{equation}
\label{eq_conv2}
(V_{2}*u^2)(x,t)\ge\frac{C\e^2 |x| \log(t-|x|) }{(t+|x|)^3}\quad \mbox{for}\quad t-|x|\ge1
\end{equation}
(see (\ref{frame-conv})),
where $u$ is the solution to \eqref{IVP} with the assumption of Theorem \ref{thm_bu}.
We note that the order except for the logarithmic function in (\ref{eq_conv2}) is the
same as the estimate (\ref{eq_conv1}) away from the light cone.
The estimate (\ref{eq_conv2}) follows from the estimates for the free wave equations:
\begin{equation}
\label{eq_free-est}
u(x,t)\ge \frac{C\e }{(t+r)(t-|x|)^{1/2}} \quad \mbox{for}\quad t-|x|\ge1,
\end{equation}
(see (\ref{first-est})).
Putting (\ref{eq_conv2}) and (\ref{eq_free-est}) to the integral of the Duhamel term in (\ref{IE_u})
and slicing the integration of the domain, we can get the first step of the iteration argument
while preserving the logarithmic function (see (\ref{it_1}) for $j=1$).

This paper is organized as follows.
In Section 2, we prepare some notations and useful lemmas.
In Section 3, we prove Theorems \ref{thm_lb} and \ref{thm_ge}.
In particular, Propositions \ref{prop_potential} and \ref{prop_potential2} are shown in this section.
In Section 4, Theorem \ref{thm_bu} is obtained.

\section{Notations, Preliminaries and Useful lemmas}
\quad \ \ In this section, we fix notations and collect some estimates which are useful when we estimate the nonlinear term.

\subsection{Notations and Preliminaries}
\quad \ \ For positive numbers $a$ and $b$, we write $a\lesssim b$ when there exists a positive constant $c$ such that $a\le cb$.
We also write $\LR{\cdot}:=1+|\cdot|$.

We introduce the solution space $X$ to the problem (\ref{IVP}) with the data $(u_0,u_1)\in C^1(\R^3)\times C(\R^3)$ given by
\EQS{\label{sols}
  X_\ga(T):=\{u\in C(\R^3\times[0,T)):\|u\|_{X_\ga(T)}<\I\},
}
where $T>0$ and the norm $\|\cdot\|_{X_\ga(T)}$ is defined by
\EQS{\label{sols_norm}
\|u\|_{X_\ga(T)}
 :=\sup_{(x,t)\in\R^3\times[0,T)}
  \LR{t+|x|}\LR{t-|x|}^{(3-\ga)/2}|u(x,t)|.
}
If there is no confusions, we write $\|\cdot\|_X=\|\cdot\|_{X_\ga(T)}$.
We also introduce a class of the initial data $Y(\ka)$ defined by
\EQQS{
  &Y(\ka)=\{(u_0,u_1)\in C^1(\R^3)\times C(\R^3);\|(u_0,u_1)\|_{Y(\ka)}<\I\},\\
  &\|(u_0,u_1)\|_{Y(\ka)}=\sup_{x\in\R^3}
    \{\LR{x}^{\ka}|u_0(x)|+\LR{x}^{\ka+1}(|\na u_0(x)|+|u_1(x)|)\}.
}
As stated in Introduction, we consider \eqref{IVP} with initial data in $Y((5-\ga)/2)$.
The integral equation on $\R^3\times [0,T)$ associated with the Cauchy problem \eqref{IVP} is
\EQS{\label{IE_u}
  u(x,t)= u^0(x,t)+L((V_\ga*u^2)u)(x,t),\quad (x,t)\in\R^3\times[0,T),
}
where $u^0$ is defined by
\EQS{\label{u^0}
  u^0(x,t):=\dt W(u_0|x,t)+W(u_1|x,t), \quad (x,t)\in\R^3\times[0,T),
}
and the integral operator $L$ on $C(\R^3\times [0,\infty))$ is defined by
\EQS{\label{L}
  L(F)(x,t):=
  \int_0^t W(F(\cdot,s)|x,t-s)ds,
}
where $F\in C(\R^3\times [0,\infty))$. Here, $W$ is
\EQS{\label{solop}
  W(\phi|x,t):=
  \frac{t}{4\pi}\int_{|\om|=1}\phi(x+t\om)dS_{\om},
}
for $\phi\in C(\R^3)$,  where $dS_{\omega}$ denotes the area element of the two dimensional unit sphere $S_2:=\{\om\in \R^3;|\om|=1\}$ in $\R^3$.

\subsection{Useful lemmas}

\quad \ \ Lemmas collected in this subsection are fundamental tools for the study of wave equations.
So, we omit some proofs.

\begin{lem}
\label{lm:Planewave}
Let $b:(0,\infty)\rightarrow\R$ be a continuous function. Then for any $\rho>0$ and $x\in \R^3$ with $r=|x|$, the identity holds:
\begin{equation}
\label{Planewave}
\begin{array}{ll}
\d \int_{|\omega|=1}b(|x+\rho \omega|)dS_\omega
\d = \frac{2\pi}{r\rho }\int_{|\rho-r|}^{\rho+r}\lambda b(\lambda)
d\lambda.
\end{array}
\end{equation}
\end{lem}
For the proof of this lemma, see Chapter I in \cite{J55} (see also \cite[Lemma 2.1]{K04}).

\begin{lem}\label{lem_2.1}
  Let $0<\de\le 1$ and $\ka>0$.

  {\rm (a)} Then there exists $C=C(\de)>0$ such that for any $(r,t)\in[0,\I)^2$,
  \EQQS{
  \int_{|t-r|}^{t+r}\frac{d\rho}{\LR{\rho}}
  \le\frac{C\min\{t^\de,r^\de\}}{\LR{t-r}^{\de}}.
  }

  {\rm (b)} Then there exists $C=C(\ka)>0$ such that for any $(r,t)\in[0,\I)^2$,
  \EQQS{
  \int_{|t-r|}^{t+r}\frac{d\rho}{\LR{\rho}^{1+\ka}}
  \le\frac{C\min\{r,t\}}{\LR{t+r}\LR{t-r}^{\ka}}.
  }
\end{lem}

\begin{proof}
  See Lemma 3.1 in \cite{T14}.
\end{proof}

The following lemma, which is a variant of (b) of Lemma \ref{lem_2.1}, is used to deduce the lower bound lifespan of the solution.

\begin{lem}\label{lem_2.2}
  Let $\ka>0$, $l\in\N\cup\{0\}$. Then there exists a positive constant $C=C(\ka,l)>0$ such that for any $(r,t)\in (0,\infty)\times[0,\I)$,
  \EQQS{
    \frac{1}{r}\int_{|t-r|}^{t+r}
      \frac{\{\log (2+\la)\}^l}{(1+\lambda)^{1+\ka}}d\la
    \le \frac{C\{\log(3+t)\}^l}{\LR{t+r} \LR{t-r}^{\ka}}.
  }
\end{lem}

\begin{proof}
  Without loss of generality, we may assume that $l\in\N$ since the case $l=0$ is classical and treated in Lemma 2.3 in \cite{K04}.
  By Lemma 2.3 in \cite{K04}, we obtain
  \EQQS{
  \frac{1}{r}\int_{|r-t|}^{r+t}
    \frac{\{\log (2+\la)\}^l}{(1+\lambda)^{1+\ka}}d\la
  \lesssim\frac{\{\log(1+\LR{t-r})\}^l}{\LR{t+r} \LR{t-r}^{\ka}}.
  }
  So, our claim is clear when $r\le 2t$ or $r\le 1$.
  Now we assume that $r\ge 2t$ and $r\ge1$.
  Let $\de>0$ be such that $\de<\min\{1,\ka\}$.
  Put
  $$f(\rho):=\frac{\{\log(2+\rho)\}^l}{(2+\rho)^{\de}}$$
  for $\rho\ge0$.
  It is easy to see that $f(\rho)$ is increasing on $(0,\rho_0)$ and is decreasing in $(\rho_0,\I)$, where $\rho_0:=\max\{e^{l/\de}-2,0\}$.
  When $t\ge\rho_0$, we have
  \EQQS{
  \frac{1}{r}\int_{r-t}^{r+t}
    \frac{\{\log (2+\la)\}^l}{(1+\lambda)^{1+\ka}}d\la
  \lesssim \frac{f(t)}{r}\int_{r-t}^{r+t}
    \frac{d\la}{(1+\lambda)^{1+\ka-\de}}
  \lesssim \frac{f(t)t}{r\LR{t+r}\LR{r-t}^{\ka-\de}}
  }
  since $r-t\ge t$.
  On the other hand, when $t\le\rho_0$, it holds that
  $0<f(0)\le f(t)\le f(\rho_0)\le Cf(0)$, where $C=C(\ka,l)>0$.
  In particular, we have $f(\rho_0)\le Cf(t)$,
  which implies that
  \EQQS{
  \frac{1}{r}\int_{r-t}^{r+t}
    \frac{\{\log (2+\la)\}^l}{(1+\lambda)^{1+\ka}}d\la
  \lesssim \frac{f(\rho_0)}{r}\int_{r-t}^{r+t}
    \frac{d\la}{(1+\lambda)^{1+\ka-\de}}
  \lesssim \frac{f(t)t}{r\LR{t+r}\LR{r-t}^{\ka-\de}}.
  }
  Since $r\ge 2t$ and $r\ge 1$, we see that
  \EQQS{
    \frac{t}{(1+t)^\de r}
    \le\frac{t^{1-\de}}{r}
    \lesssim\frac{(1+r-t)^{1-\de}}{1+r-t}
    =\frac{1}{\LR{r-t}^{\de}},
  }
  which completes the proof.
\end{proof}

The following lemma is Lemma 2.2 in Kubo and Ohta \cite{KO05}

\begin{lem}
\label{lem:KO}
Let $\kappa\in \R$ and $C=2/\max\{\kappa,1\}$.
Then
\begin{equation}
\label{est:KO}
\int_{t-r}^{t+r}\frac{d\rho}{\rho^{1+\kappa}}\ge \frac{Cr}{(t+r)(t-r)^{\kappa}}
\end{equation}
holds for $t>r>0$.
\end{lem}

\begin{lem}[Estimates for $W$]\label{lem4.2}
\begin{enumerate}
\item
  Let $T>0$, $\Phi \in C(\R^3)$ and $\phi \in C([0,\I))$.
  If the inequality $|\Phi(x)|\le \phi(|x|)$ holds for any $x\in\R^3$, then the estimate
  \EQS{
  |W(\Phi|x,t)|\le \frac{1}{2r}\int_{|r-t|}^{r+t}\la \phi(\la) d\la
  }
holds for any $(x,t)\in\R^3\times [0,T)$ with $r=|x|$.
\item
  Let $T>0$, $\Psi\in C(\R^3\times [0,T))$ and $\psi\in C([0,\I)\times [0,T))$. We assume that the estimate $|\Psi(x,t)|\le \psi(|x|,t)$ holds for any $(x,t)\in \R^3\times [0,T)$.
  Then the estimate
  \begin{equation}
  \label{3-4-1}
    \left|\int_0^t W\left(\Psi(\cdot,s)|x,t-s\right)ds\right|
  \le \frac{1}{2r}\iint_{D(r,t)}\la \psi(\la,s)d\la ds,
  \end{equation}
  holds for any $(x,t)\in\R^3\times [0,T)$ with $r=|x|$, where $D(r,t)$ is defined by
  \begin{equation}
\label{def:D}
 D(r,t):=\left\{(\la,s)\in[0,\I)^2\ :\ s\in [0,t], |r-t+s|\le\la\le r+t-s\right\}.
  \end{equation}
\end{enumerate}
\end{lem}
The following lemma is useful to estimate the nonlinear term in the proof of Theorem \ref{thm_bu}.
\begin{lem}
\label{lem:est-du}
Let $\Psi\in C(\R^3\times[0,T))$ and $\psi \in C ([0,\infty)\times[0,T))$. Assume that
$\Psi(x,t)\ge \psi(|x|,t)\ge0$ for $(x,t)\in\R^3\times[0,T)$ holds. Then we have
\begin{equation}
\label{est-du}
L(\Psi)(x,t)\ge \frac{1}{2r} \iint_{D(r,t)}\lambda\psi(\lambda,s) d\lambda ds,
\end{equation}
where $r=|x|$ and $D(r,t)$ is the one in (\ref{def:D})
\end{lem}
\begin{proof}
From the definition of $L$ in (\ref{L}), (\ref{solop}) and (\ref{Planewave}), we get
\[
\begin{array}{llll}
L(\Psi)(x,t)&=\d\frac{1}{4\pi}\int_{0}^{t}(t-s)\int_{|\omega|=1}\Psi(x+(t-s)\omega,s)dS_{\omega}ds\\
&\d\ge\frac{1}{4\pi}\int_{0}^{t}(t-s)\int_{|\omega|=1}\psi(|x+(t-s)\omega|,s)dS_{\omega}ds\\
&\d=\frac{1}{2r}\int_{0}^{t} \int_{|r-t+s|}^{r+t-s}\lambda \psi(\lambda,s)d\lambda ds.
\end{array}
\]
The proof is completed.
\end{proof}

\section{Lower bound of lifespan for $\ga=2$ and global existence for $2<\ga<3$}

\quad \ \ In this section, we show Theorems \ref{thm_lb} and \ref{thm_ge}.
The following lemma is the estimate for the free solution $u^0(x,t)$, which is defined by \eqref{u^0}.
For the proof, see (2.14) in \cite{K04}.

\begin{lem}[Estimates for free solutions]\label{lem_free}
  Let $\nu>0$ and $T>0$.
  Then there exists $C=C(\nu)>0$ such that
  \EQQS{
    \sup_{(x,t)\in\R^3\times[0,T)}
    \LR{t+|x|}\LR{t-|x|}^{\nu}|u^0(x,t)|\le C\|(u_0,u_1)\|_{Y(\nu)}
  }
  for any $(u_0,u_1)\in Y(\nu)$.
\end{lem}

\subsection{Multilinear estimates}

\quad \ \ The main estimates of the present paper are the following two propositions.

\begin{prop}\label{prop_potential}
  Let $\ga=2$ and $T>0$.
  Then there exists $C>0$ such that
  \EQQS{
    |(V_2*(u_1 u_2))(y,s)|
    \le\frac{C(1+\log\LR{s+|y|})}{\LR{s+|y|}^{7/4}\LR{s-|y|}^{1/4}}\|u_1\|_{X}\|u_2\|_{X}
  }
  for any $u_1,u_2\in X_2(T)$ and $(y,s)\in\R^3\times[0,T)$.
\end{prop}

\begin{prop}\label{prop_potential2}
  Let $2<\ga<3$ and $T>0$.
  Then there exists $C=C(\ga)>0$ such that
  \EQQS{
    |(V_\ga*(u_1 u_2))(y,s)|
    \le\frac{C\|u_1\|_{X}\|u_2\|_{X}}{\LR{s+|y|}^{(5+\ga)/4}\LR{s-|y|}^{(3-\ga)/4}}
  }
  for any $u_1,u_2\in X_\ga(T)$ and $(y,s)\in\R^3\times[0,T)$.
\end{prop}

\begin{proof}[Proof of Proposition \ref{prop_potential}]
  We follow the argument of Proposition 2.2 in \cite{T14}.
  As stated in Introduction, we give up earning the full power $\LR{s+|y|}^{-2}$ since we need to eliminate a logarithmic growth in (ii) of Case 1 and (ii) of Case 3.
  For the sake of the simplicity, we put $M:=\|u_1\|_{X}\|u_2\|_{X}$.
  Set
  \EQQS{
    &|(V_2*(u_1 u_2))(y,s)|\\
    &\le M\int_{\R^3}|y-z|^{-2}\LR{s+|z|}^{-2}\LR{s-|z|}^{-1}dz\\
    &\le M\Biggr(\int_{|y-z|\le1/2}+\int_{|y-z|\ge1/2}\Biggr)|y-z|^{-2}\LR{s+|z|}^{-2}\LR{s-|z|}^{-1}dz
    =:A+B.
  }
  It is easy to see that $1+s+|z|\ge (1+s+|y|)/2$ and $1+|s-|z||\ge (1+|s-|y||)/2$ when $|y-z|\le 1/2$.
  Then, we have
  \EQQS{
    A\lesssim\frac{M}{\LR{s+\la}^2\LR{s-\la}}\int_{|y-z|\le 1/2}\frac{dz}{|y-z|^{2}}
    \lesssim \frac{M}{\LR{s+\la}^2\LR{s-\la}},
  }
  where $\la:=|y|$.
  If $|y-z|\ge1/2$, we see that $|y-z|\ge (1+|y-z|)/3$.
  So, we have
  \EQQS{
    B
    &\lesssim M\int_{|y-z|\ge 1/2}
      \frac{dz}{\LR{y-z}^2\LR{s+|z|}^2\LR{s-|z|}}\\
    &\lesssim\frac{M}{\la}
      \int_0^\I \frac{\eta}{\LR{s+\eta}^2\LR{s-\eta}}
      \int_{|\la-\eta|}^{\la+\eta}\frac{\rho}{\LR{\rho}^2}d\rho d\eta,
  }
  where we used Lemma \ref{lm:Planewave} in the last inequality.
  We devide the proof into the three cases where
  $\{(\la,s)\in[0,\infty)\times[0,T);\la\ge 1, \la\ge s\}$,
  $\{(\la,s)\in[0,\infty)\times[0,T);s\le\la\le 1\}$ and
  $\{(\la,s)\in[0,\infty)\times[0,T);\la\le s\}$.\\
  {\bf Case 1.} $\la\ge 1$ and $\la\ge s$.
  Set
  \EQQS{
    B
    &\lesssim\frac{M}{\la}\Biggr(\int_0^s+\int_s^\la+\int_\la^\I\Biggr)
    \frac{\eta}{\LR{s+\eta}^2\LR{s-\eta}}
    \int_{|\la-\eta|}^{\la+\eta}\frac{\rho}{\LR{\rho}^2}d\rho d\eta\\
    &=:B_1+B_2+B_3.
  }
  We further devide the proof into two cases where $\la\ge 2s$ and $s\le \la\le 2s$.\\
  (i) $\la\ge 2s$.
  We use (a) of Lemma \ref{lem_2.1} with $\de=1$ to evaluate $B_1$.
  Observe that
  \EQQS{
    B_1
    &\lesssim\frac{M}{\la}\int_0^s
     \frac{\eta^2}{\LR{s+\eta}^2\LR{s-\eta}\LR{\la-\eta}}d\eta\\
    &\lesssim\frac{M}{\LR{s+\la}^2}
     \int_0^s\frac{d\eta}{\LR{s-\eta}}
    \lesssim\frac{\log\LR{s+\la}}{\LR{s+\la}^2}M
  }
  since $\la-\eta\ge\la-s\ge \la/2$.
  Next, we estimate $B_2$.
  Note that $\eta\le(s+\la)/2$ implies $\la-\eta\ge(\la-s)/2\ge\la/4$, and $\eta\ge(s+\la)/2$ implies $\eta-s\ge (\la-s)/2\ge\la/4$. This tohether (a) of Lemma \ref{lem_2.1} with $\de=1$ shows that
  \EQQS{
    B_2
    &\lesssim\frac{M}{\LR{\la}}
    \Biggr(\int_s^{(s+\la)/2}+\int_{(s+\la)/2}^\la\Biggr)
    \frac{\eta^2}{\LR{s+\eta}^2 \LR{\eta-s} \LR{\la-\eta}}d\eta\\
    &\lesssim\frac{M}{\LR{s+\la}^2}
      \int_s^{(s+\la)/2}\frac{d\eta}{\LR{\eta-s}}
    +\frac{M}{\LR{s+\la}^{2}}
      \int_{(s+\la)/2}^\la\frac{d\eta}{\LR{\la-\eta}}\\
    &\lesssim\frac{\log\LR{s+\la}}{\LR{s+\la}^2}M.
  }
  In order to estimate $B_3$, we use (a) of Lemma \ref{lem_2.1} with $\de=1$, and we have
  \EQQS{
    B_3
    &\lesssim\frac{M}{\LR{s+\la}}
      \Biggr(\int_\la^{2\la}+\int_{2\la}^\I\Biggr)
      \frac{\eta\la}{\LR{s+\eta}^2 \LR{\eta-s} \LR{\eta-\la}}d\eta\\
    &\lesssim\frac{M}{\LR{s+\la}^{2}}
      \int_\la^{2\la}\frac{d\eta}{\LR{\eta-\la}}
     +\frac{M}{\LR{s+\la}}\int_{2\la}^\I
      \frac{d\eta}{\LR{\eta-\la}^2}
    \lesssim\frac{1+\log\LR{s+\la}}{\LR{s+\la}^2}M
  }
  since $\eta-s\ge\la-s\ge\la/2$ when $\eta\ge \la$.\\
  (ii) $s\le \la\le 2s$.
  We use (a) of Lemma \ref{lem_2.1} with $\de=1/4$ to estimate $B_1,B_2$ and $B_3$.
  Observe that
  \EQQS{
    B_1
    &\lesssim\frac{M}{\la}\int_0^s
      \frac{\eta^{5/4}}{\LR{s+\eta}^2\LR{s-\eta}\LR{\la-\eta}^{1/4}}d\eta\\
    &\lesssim\frac{M}{\LR{s+\la}^{7/4}\LR{\la-s}^{1/4}}
      \int_0^s\frac{d\eta}{\LR{s-\eta}}
    \lesssim \frac{\log\LR{s+\la}}{\LR{s+\la}^{7/4}\LR{\la-s}^{1/4}}M.
  }
  Similarly, we obtain
  \EQQS{
    B_2
    &\lesssim\frac{M}{\LR{s+\la}^{7/4}}
      \Biggr(\int_s^{(s+\la)/2}+\int_{(s+\la)/2}^\la\Biggr)\frac{d\eta}{\LR{\eta-s}\LR{\la-\eta}^{1/4}}\\
    &\lesssim\frac{M}{\LR{s+\la}^{7/4}\LR{\la-s}^{1/4}}
      \int_s^{(s+\la)/2}\frac{d\eta}{\LR{\eta-s}}\\
    &\quad+\frac{M}{\LR{s+\la}^{7/4}\LR{\la-s}}
      \int_{(s+\la)/2}^\la\frac{d\eta}{\LR{\la-\eta}^{1/4}}
    \lesssim \frac{1+\log\LR{s+\la}}{\LR{s+\la}^{7/4}\LR{\la-s}^{1/4}}M.
  }
  On the other hand, we note that $2\la-s\ge\la$.
  It then follows that
  \EQQS{
    B_3
    &\lesssim\frac{M}{\la}\Biggr(\int_\la^{2\la-s}+\int_{2\la-s}^\I\Biggr)
      \frac{\eta\la^{1/4}}{\LR{s+\eta}^2\LR{\eta-s}\LR{\eta-\la}^{1/4}}d\eta\\
    &\lesssim\frac{M}{\LR{s+\la}^{7/4}\LR{\la-s}}
      \int_\la^{2\la-s}\frac{d\eta}{\LR{\eta-\la}^{1/4}}
     +\frac{M}{\LR{s+\la}^{7/4}}
      \int_{2\la-s}^\I\frac{d\eta}{\LR{\eta-\la}^{5/4}}\\
    &\lesssim\frac{M}{\LR{s+\la}^{7/4}\LR{\la-s}^{1/4}}.
  }
  This completes the proof for the case $\la\ge1$ and $\la\ge s$.\\
  {\bf Case 2.} $s\le\la\le1$.
  In this case we can obtain the desired inequality easily.
  Indeed, we see from (a) of Lemma \ref{lem_2.1} with $\de=1$ that
  \EQQS{
    B
    &\lesssim\frac{M}{\la}
      \int_0^\I \frac{\eta\min\{\eta,\la\}}{\LR{s+\eta}^2 \LR{s-\eta} \LR{\la-\eta}}d\eta\\
    &\lesssim M\int_0^{2s}d\eta
      +M\int_{2s}^{\I}\frac{d\eta}{\LR{s+\eta}^2}
    \lesssim M
    \lesssim\frac{M}{\LR{s+\la}^2}.
  }
  Here we used $\eta-s\ge\eta/2$ when $\eta\ge2s$.\\
  {\bf Case 3.} $s\ge\la$.
  As in the Case 1, we devide the integral into three pieses: set
  \EQQS{
    B
    &\lesssim\frac{M}{\la}\Biggr(\int_0^\la+\int_\la^s+\int_s^\I\Biggr)
    \frac{\eta}{\LR{s+\eta}^2\LR{s-\eta}}
    \int_{|\la-\eta|}^{\la+\eta}\frac{\rho}{\LR{\rho}^2}d\rho d\eta\\
    &=:B_1+B_2+B_3.
  }
  We also further devide the proof into two cases where $s\ge 2\la$ and $\la\le s\le 2\la$.\\
  (i) $s\ge 2\la$.
  We see from (a) of Lemma \ref{lem_2.1} with $\de=1$ that
  \EQQS{
    B_1
    &\lesssim\frac{M}{\la}
      \int_0^\la\frac{\eta^{2}}{\LR{s+\eta}^2 \LR{s-\eta} \LR{\la-\eta}}d\eta\\
    &\lesssim\frac{M}{\LR{s+\la}^{2}}
      \int_0^\la\frac{d\eta}{\LR{\la-\eta}}
    \lesssim\frac{\log\LR{s+\la}}{\LR{s+\la}^{2}}M.
  }
  Next, by (a) of Lemma \ref{lem_2.1} with $\de=1$, we have
  \EQQS{
    B_2
    &\lesssim M \Biggr(\int_\la^{(s+\la)/2}+\int_{(s+\la)/2}^s\Biggr)
      \frac{\eta}{\LR{s+\eta}^2 \LR{s-\eta} \LR{\eta-\la}}d\eta\\
    &\lesssim \frac{M}{\LR{s+\la}^2} \int_\la^{(s+\la)/2}
      \frac{d\eta}{\LR{\eta-\la}}
     +\frac{M}{\LR{s+\la}^2} \int_{(s+\la)/2}^s
      \frac{d\eta}{\LR{s-\eta}}\\
    &\lesssim\frac{\log\LR{s+\la}}{\LR{s+\la}^{2}}M
  }
  since $s-\eta\ge(s-\la)/2\ge s/4$ when $\eta\le (s+\la)/2$, and $\eta-\la\ge(s-\la)/2\ge s/4$ when $\eta\ge (s+\la)/2$.
  For $B_3$, we also (a) of use Lemma \ref{lem_2.1} with $\de=1$ so that
  \EQQS{
    B_3
    &\lesssim M \Biggr(\int_s^{2s}+\int_{2s}^\I\Biggr)
      \frac{\eta}{\LR{s+\eta}^2 \LR{\eta-s} \LR{\eta-\la}}d\eta\\
    &\lesssim \frac{M}{\LR{s+\la}^2}
      \int_s^{2s}\frac{d\eta}{\LR{\eta-s}}
     +M \int_{2s}^\I \frac{d\eta}{\LR{\eta-s}^3}
    \lesssim\frac{1+\log\LR{s+\la}}{\LR{s+\la}^{2}}M
  }
  since $\eta-\la\ge s-\la\ge s/2$ and $\eta-\la\ge\eta/2$ when $\eta\ge s$.\\
  (ii) $\la\le s\le 2\la$.
  We may assume that $\la\ge 1$, otherwise the proof is identical with Case 2.
  We use (a) of Lemma \ref{lem_2.1} with $\de=1/4$ to estimate $B_1,B_2$ and $B_3$.
  Observe that
  \EQQS{
    B_1
    &\lesssim\frac{M}{\la}\int_0^\la
      \frac{\eta^{5/4}}{\LR{s+\eta}^2 \LR{s-\eta} \LR{\la-\eta}^{1/4}}d\eta\\
    &\lesssim\frac{M}{\LR{s+\la}^{7/4} \LR{s-\la}^{1/4}}
      \int_0^\la \frac{d\eta}{\LR{\la-\eta}}
    \lesssim\frac{\log\LR{s+\la}}{\LR{s+\la}^{7/4} \LR{s-\la}^{1/4}}M.
  }
  Here, we used $\LR{s-\eta}=\LR{s-\eta}^{1/4}\LR{s-\eta}^{3/4}\ge\LR{s-\la}^{1/4}\LR{\la-\eta}^{3/4}$ for $\eta\le\la$.
  Similarly, we have
  \EQQS{
    B_2
    &\lesssim\frac{M}{\LR{s+\la}^{7/4}}
      \Biggr(\int_\la^{(s+\la)/2}+\int_{(s+\la)/2}^s\Biggr)\frac{d\eta}{\LR{s-\eta}\LR{\eta-\la}^{1/4}}\\
    &\lesssim\frac{M}{\LR{s+\la}^{7/4}\LR{s-\la}^{1/4}}
      \int_\la^{(s+\la)/2}\frac{d\eta}{\LR{\eta-\la}}\\
    &\quad+\frac{M}{\LR{s+\la}^{7/4}\LR{s-\la}^{1/4}}
      \int_{(s+\la)/2}^s\frac{d\eta}{\LR{s-\eta}}
    \lesssim \frac{1+\log\LR{s+\la}}{\LR{s+\la}^{7/4}\LR{\la-s}^{1/4}}M.
  }
  Finally, we see that
  \EQQS{
    B_3
    &\lesssim\frac{M}{\la}\Biggr(\int_s^{2s-\la}+\int_{2s-\la}^\I\Biggr)
      \frac{\eta\la^{1/4}}{\LR{s+\eta}^2\LR{\eta-s}\LR{\eta-\la}^{1/4}}d\eta\\
    &\lesssim\frac{M}{\LR{s+\la}^{7/4}\LR{s-\la}^{1/4}}
      \int_s^{2s-\la}\frac{d\eta}{\LR{\eta-s}}
     +\frac{M}{\LR{s+\la}^{7/4}}
      \int_{2s-\la}^\I\frac{d\eta}{\LR{\eta-s}^{5/4}}\\
    &\lesssim\frac{1+\log\LR{s+\la}}{\LR{s+\la}^{7/4}\LR{s-\la}^{1/4}}M.
  }
  since $\la\ge1$, $2s-\la\ge s$ and $\eta-\la\ge\eta-s$.
  This concludes the proof.
\end{proof}

Next, we show Proposition \ref{prop_potential2}.
The proof is similar to that of Proposition \ref{prop_potential}.
Similarly to Proposition \ref{prop_potential}, estimates from (ii) of Case 1 and (ii) of Case 3 are the worst.

\begin{proof}[Proof of Proposition \ref{prop_potential2}]
  Put $M:=\|u_1\|_{X}\|u_2\|_{X}$.
  We set
  \EQQS{
    &|(V_\ga*(u_1 u_2))(y,s)|\\
    &\le M\int_{\R^3}|y-z|^{-\ga}\LR{s+|z|}^{-2}\LR{s-|z|}^{-(3-\ga)}dz\\
    &\le M\Biggr(\int_{|y-z|\le1/2}+\int_{|y-z|\ge1/2}\Biggr)|y-z|^{-\ga}\LR{s+|z|}^{-2}\LR{s-|z|}^{-(3-\ga)}dz
    =:A+B.
  }
  As in the proof of Proposition \ref{prop_potential}, we can obtain $A\lesssim \LR{s+\la}^{-2}\LR{s-\la}^{-(3-\ga)}M$, where $\la:=|y|$.
  If $|y-z|\ge1/2$, we also have
  \EQQS{
    B
    &\lesssim\frac{M}{\la}\int_0^\I\frac{\eta}{\LR{s+\eta}^2\LR{s-\eta}^{3-\ga}}\int_{|\la-\eta|}^{\la+\eta}\frac{\rho}{\LR{\rho}^\ga}d\rho d\eta\\
    &\lesssim\frac{M}{\la}
      \int_0^\I \frac{\eta\min\{\eta,\la\}}{\LR{s+\eta}^2\LR{s-\eta}^{3-\ga}\LR{\la+\eta}\LR{\la-\eta}^{\ga-2}}d\eta,
  }
  where we used (b) of Lemma \ref{lem_2.1} in the last inequality since $\ga-2>0$.
  As in the proof of Proposition \ref{prop_potential}, we devide the proof into the three cases where
  $\{(\la,s)\in[0,\infty)\times[0,T);\la\ge 1, \la\ge s\}$,
  $\{(\la,s)\in[0,\infty)\times[0,T);s\le\la\le 1\}$ and
  $\{(\la,s)\in[0,\infty)\times[0,T);\la\le s\}$.\\
  {\bf Case 1.} $\la\ge 1$ and $\la\ge s$.
  Set
  \EQQS{
    B
    &\lesssim\frac{M}{\la}\Biggr(\int_0^s+\int_s^\la+\int_\la^\I\Biggr)
    \frac{\eta\min\{\eta,\la\}}{\LR{s+\eta}^2\LR{s-\eta}^{3-\ga}\LR{\la+\eta}\LR{\la-\eta}^{\ga-2}}d\eta\\
    &=:B_1+B_2+B_3.
  }
  We further devide the proof into two cases where $\la\ge 2s$ and $s\le \la\le 2s$.\\
  (i) $\la\ge 2s$.
  It is easy to see that
  \EQQS{
    B_1
    \lesssim\frac{M}{\LR{s+\la}^2\LR{\la-s}^{\ga-2}}
      \int_0^s\frac{d\eta}{\LR{s-\eta}^{3-\ga}}
    \lesssim\frac{M}{\LR{s+\la}^2}
  }
  since $0\le s\le\la-s$.
  Next, we see that
  \EQQS{
    B_2
    &\lesssim\frac{M}{\LR{s+\la}^2}
    \Biggr(\int_s^{(s+\la)/2}+\int_{(s+\la)/2}^\la\Biggr)
    \frac{d\eta}{\LR{\eta-s}^{3-\ga} \LR{\la-\eta}^{\ga-2}}\\
    &\lesssim\frac{M}{\LR{s+\la}^2 \LR{\la-s}^{\ga-2}}
      \int_s^{(s+\la)/2}\frac{d\eta}{\LR{\eta-s}^{3-\ga}}\\
    &\quad+\frac{M}{\LR{s+\la}^{2} \LR{\la-s}^{3-\ga}}
      \int_{(s+\la)/2}^\la\frac{d\eta}{\LR{\la-\eta}^{\ga-2}}
    \lesssim\frac{M}{\LR{s+\la}^2}.
  }
  On the other hand, it holds that
  \EQQS{
    B_3
    &\le M
      \Biggr(\int_\la^{2\la}+\int_{2\la}^\I\Biggr)
      \frac{\eta}{\LR{s+\eta}^2 \LR{\eta-s}^{3-\ga} \LR{\la+\eta} \LR{\eta-\la}^{\ga-2}}d\eta\\
    &\lesssim\frac{M}{\LR{s+\la}^{5-\ga}}
      \int_\la^{2\la}\frac{d\eta}{\LR{\eta-\la}^{\ga-2}}
     +\frac{M}{\LR{s+\la}}\int_{2\la}^\I
      \frac{d\eta}{\LR{\eta-\la}^2}
    \lesssim\frac{M}{\LR{s+\la}^2}
  }
  since $\eta-s\ge\la-s\ge\la/2$ when $\eta\ge\la$.
  For the second term, we used $\eta-s\ge\eta-\la$ and $s+\eta\ge \eta-\la$.\\
  (ii) $s\le \la\le 2s$.
  Put
  \EQS{\label{def_theta}
    \th:=\max\Biggr\{\frac{5\ga-11}{4(\ga-2)},0\Biggr\}.
  }
  Note that $0\le\th<1$, and $3-\ga+\th(\ga-2)<1$ since $\ga>2$.
  Using $\LR{\la-\eta}\ge \LR{s-\eta}^{\th} \LR{\la-\eta}^{1-\th}\ge \LR{s-\eta}^{\th} \LR{\la-s}^{1-\th}$ for $0\le \eta\le s$, we have
  \EQQS{
    B_1
    &\lesssim\frac{M}{\LR{s+\la}^2}\int_0^s\frac{d\eta}{\LR{s-\eta}^{3-\ga} \LR{\la-\eta}^{\ga-2}}\\
    &\le\frac{M}{\LR{s+\la}^2 \LR{\la-s}^{(1-\th)(\ga-2)}}\int_0^s\frac{d\eta}{\LR{s-\eta}^{3-\ga+\th(\ga-2)}}\\
    &\lesssim\frac{M\LR{s}^{(1-\th)(\ga-2)-(3-\ga)/4}}{\LR{s+\la}^{(5+\ga)/4}\LR{\la-s}^{(1-\th)(\ga-2)}}
    \lesssim\frac{M}{\LR{s+\la}^{(5+\ga)/4}\LR{\la-s}^{(3-\ga)/4}}
  }
  since $(1-\th)(\ga-2)-(3-\ga)/4\le 0$ and $0\le\la-s\le s$.
  On the other hand, we can estimate $B_2$ by the same way as in Case 1 (i).
  Next, we evaluate $B_3$:
  \EQQS{
  B_3
  \lesssim\frac{M}{\LR{s+\la}}
    \Biggr(\int_\la^{2\la}+\int_{2\la}^\I\Biggr)
    \frac{\eta\la}{\LR{s+\eta}^2 \LR{\eta-s}^{3-\ga} \LR{\la+\eta} \LR{\eta-\la}^{\ga-2}}d\eta.
  }
  The second term in the right hand side can be estimated by the same way as in Case 1 (i).
  Thus, we focus on the first term, which we denote by $B_{31}$.
  Using $\LR{\eta-s}\ge \LR{\eta-s}^{1/4}\LR{\eta-\la}^{3/4}\ge\LR{\la-s}^{1/4}\LR{\eta-\la}^{3/4}$ for $\eta\ge\la$,
  we have
  \EQQS{
  B_{31}
  &\lesssim\frac{M}{\LR{s+\la}^2}
    \int_\la^{2\la}
    \frac{d\eta}{\LR{\eta-s}^{3-\ga} \LR{\eta-\la}^{\ga-2}}\\
  &\le\frac{M}{\LR{s+\la}^2 \LR{\la-s}^{(3-\ga)/4}}
    \int_\la^{2\la}\frac{d\eta}{\LR{\eta-\la}^{(\ga+1)/4}}\\
  &\lesssim\frac{M}{\LR{s+\la}^{(5+\ga)/4}\LR{\la-s}^{(3-\ga)/4}}
  }
  since $(\ga+1)/4<1$.\\
  {\bf Case 2} $s\le\la\le1$.
  The proof is identical with that of Proposition \ref{prop_potential}.\\
  {\bf Case 3} $s\ge\la$.
  As in the Case 1, we devide the integral into three pieses: set
  \EQQS{
    B
    &\lesssim\frac{M}{\la}\Biggr(\int_0^\la+\int_\la^s+\int_s^\I\Biggr)
    \frac{\eta\min\{\eta,\la\}}{\LR{s+\eta}^2\LR{s-\eta}^{3-\ga}\LR{\la+\eta}\LR{\la-\eta}^{\ga-2}}d\eta\\
    &=:B_1+B_2+B_3.
  }
  We also further devide the proof into two cases where $s\ge 2\la$ and $\la\le s\le 2\la$.\\
  (i) $s\ge 2\la$.
  It is easy to see that
  \EQQS{
    B_1
    \lesssim\frac{M}{\LR{s+\la}^2 \LR{s-\la}^{3-\ga}}
      \int_0^\la\frac{d\eta}{\LR{\la-\eta}^{\ga-2}}
    \lesssim\frac{M}{\LR{s+\la}^2}.
  }
  since $\la\le s-\la$.
  On the other hand,
  \EQQS{
    B_2
    &\lesssim\frac{M}{\LR{s+\la}^2}
    \Biggr(\int_\la^{(s+\la)/2}+\int_{(s+\la)/2}^s\Biggr)
    \frac{d\eta}{\LR{s-\eta}^{3-\ga} \LR{\eta-\la}^{\ga-2}}\\
    &\lesssim\frac{M}{\LR{s+\la}^2 \LR{s-\la}^{3-\ga}}
      \int_\la^{(s+\la)/2}\frac{d\eta}{\LR{\eta-\la}^{\ga-2}}\\
    &\quad+\frac{M}{\LR{s+\la}^{2} \LR{s-\la}^{\ga-2}}
      \int_{(s+\la)/2}^s\frac{d\eta}{\LR{s-\eta}^{3-\ga}}
    \lesssim\frac{M}{\LR{s+\la}^2}.
  }
  Next, we see that
  \EQQS{
    B_3
    &\lesssim M
      \Biggr(\int_s^{2s}+\int_{2s}^\I\Biggr)
      \frac{\eta}{\LR{s+\eta}^2 \LR{\eta-s}^{3-\ga} \LR{\la+\eta} \LR{\eta-\la}^{\ga-2}}d\eta\\
    &\lesssim\frac{M}{\LR{s+\la}^{\ga}}
      \int_s^{2s}\frac{d\eta}{\LR{\eta-s}^{3-\ga}}
     +\frac{M}{\LR{s+\la}}\int_{2s}^\I
      \frac{d\eta}{\LR{\eta-s}^2}
    \lesssim\frac{M}{\LR{s+\la}^2}
  }
  since $\eta-\la\ge s-\la\ge s/2$.
  For the second term, we used $\eta-\la\ge\eta-s$ and $s+\eta\ge \eta-s$.\\
  (ii) $\la\le s\le 2\la$.
  We may assume that $\la\ge 1$, otherwise the proof is identical with Case 2.
  Using $\LR{s-\eta}\ge \LR{s-\eta}^{1/4}\LR{\la-\eta}^{3/4}\ge \LR{s-\la}^{1/4}\LR{\la-\eta}^{3/4}$ for $0\le \eta\le \la$,
  we obtain
  \EQQS{
    B_1
    &\lesssim\frac{M}{\LR{s+\la}^2}
      \int_0^\la\frac{d\eta}{\LR{s-\eta}^{3-\ga}\LR{\la-\eta}^{\ga-2}}\\
    &\lesssim\frac{M}{\LR{s+\la}^2\LR{s-\la}^{(3-\ga)/4}}
      \int_0^\la\frac{d\eta}{\LR{\la-\eta}^{(\ga+1)/4}}\\
    &\lesssim\frac{M}{\LR{s+\la}^{(5+\ga)/4}\LR{s-\la}^{(3-\ga)/4}}
  }
  since $(\ga+1)/4<1$.
  On the other hand, we can estimate $B_2$ by the same way as in Case 3 (i).
  Next, we evaluate $B_3$:
  \EQQS{
  B_3
  \lesssim M
    \Biggr(\int_s^{2s}+\int_{2s}^\I\Biggr)
    \frac{\eta}{\LR{s+\eta}^2 \LR{\eta-s}^{3-\ga} \LR{\la+\eta} \LR{\eta-\la}^{\ga-2}}d\eta.
  }
  The second term in the right hand side can be estimated by the same way as in Case 3 (i).
  Thus, we focus on the first term, which we denote by $B_{31}$.
  Using $\LR{\eta-\la}\ge \LR{\eta-\la}^{1-\th}\LR{\eta-s}^\th \ge \LR{s-\la}^{1-\th}\LR{\eta-s}^\th$ for $s\le \eta\le 2s$ ($\th$ is defined by \eqref{def_theta}),
  we obtain
  \EQQS{
    B_{31}
    &\lesssim\frac{M}{\LR{s+\la}^2}\int_s^{2s}\frac{d\eta}{\LR{\eta-s}^{3-\ga}\LR{\eta-\la}^{\ga-2}}\\
    &\lesssim\frac{M}{\LR{s+\la}^2 \LR{s-\la}^{(1-\th)(\ga-2)}}
      \int_s^{2s}\frac{d\eta}{\LR{\eta-s}^{3-\ga+\th(\ga-2)}}\\
    &\lesssim\frac{M\LR{s}^{(1-\th)(\ga-2)-(3-\ga)/4}}{\LR{s+\la}^{(5+\ga)/4} \LR{s-\la}^{(1-\th)(\ga-2)}}
    \lesssim\frac{M}{\LR{s+\la}^{(5+\ga)/4} \LR{s-\la}^{(3-\ga)/4}}
  }
  since $3-\ga+\th(\ga-2)<1$.
  This completes the proof.
\end{proof}

\subsection{Proof of Theorems \ref{thm_lb} and \ref{thm_ge}}

\quad \ \ Using Propositions \ref{prop_potential} and \ref{prop_potential2}, we obtain the estimate for the Duhamel term in \eqref{IE_u}.

\begin{prop}\label{prop_duhamel}
Let $2\le\ga<3$, $T>0$ and $L$ be the integral operator on $C(\R^3\times[0,T))$ given by {\rm (\ref{L})}. Then there exists a positive constant $C>0$ such that
\EQQ{
  \|L((V*(u_1u_2))u_3)\|_{X}
  \le C D_\ga(T)\prod_{i=1}^3\|u_i\|_{X}
}
for any $u_1,u_2,u_3\in X_\ga(T)$, where $D_\ga(T)$ is defined by
\EQQS{
  D_\ga(T):=
  \begin{cases}
    1+\log(3+T),&\ga=2\\
    1,&2<\ga<3.
  \end{cases}
}
\end{prop}

\begin{proof}
  For the sake of simplicity, we put $M:=\prod_{i=1}^3\|u_i\|_{X}$.
  Let $l(x)$ be a function on $\R$ such that $l(2)=1$ and $l(x)=0$ for $x\neq2$.
  We see from Propositions \ref{prop_potential} and \ref{prop_potential2} that
  \EQQS{
    |((V_\ga*(u_1u_2))u_3)(x,t)|
    \lesssim \frac{1+(\log\LR{t+|x|})^{l(\ga)}}{\LR{t+|x|}^{(9+\ga)/4} \LR{t-|x|}^{3(3-\ga)/4}}M
  }
  for any $(x,t)\in\R^3\times[0,T)$.
  Put $r:=|x|$.
  Then, Lemma \ref{lem4.2} shows that
  \EQQS{
    |L((V_\ga*(u_1u_2))u_3)(x,t)|
    \lesssim\frac{M}{r}\iint_{D(r,t)}
      \frac{\la(1+(\log\LR{s+\la})^{l(\ga)})}{\LR{s+\la}^{(9+\ga)/4} \LR{s-\la}^{3(3-\ga)/4}}dsd\la.
  }
  Changing of veriables $\al:=\la+s$ and $\be:=\la-s$, we have
  \EQQS{
    |L((V_\ga*(u_1u_2))u_3)(x,t)|
    &\lesssim\frac{M}{r}
      \int_{|r-t|}^{r+t}\frac{1+(\log\LR{\al})^{l(\ga)}}{\LR{\al}^{(5+\ga)/4}}
      \int_{r-t}^\al\frac{d\be}{\LR{\be}^{3(3-\ga)/4}}d\al\\
    &\lesssim\frac{M}{r}
      \int_{|r-t|}^{r+t}\frac{1+(\log\LR{\al})^{l(\ga)}}{\LR{\al}^{(5-\ga)/2}}d\al\\
    &\lesssim\frac{1+(\log(3+T))^{l(\ga)}}{\LR{t+r}\LR{t-r}^{(3-\ga)/2}}M
  }
  since $3(3-\ga)/4<1$ for $2\le\ga<3$.
  Here, we applied Lemma \ref{lem_2.2} in the last inequality,
  which completes the proof.
\end{proof}

We are ready to prove Theorem \ref{thm_ge}.

\begin{proof}[Proof of Theorem \ref{thm_ge}]
  Let $(u_0,u_1)\in Y((5-\ga)/2)$ and set $M:=\|(u_0,u_1)\|_{Y((5-\ga)/2)}$.
  Let $C_0,C_1$ be defined by Lemma \ref{lem_free}, Proposition \ref{prop_duhamel}, respectively.
  For $\e>0$, we put
  $$X(\e):=\{u\in C(\R^3\times[0,\infty));\|u\|_{X_\ga(\I)}\le 2C_0M\e\}.$$
  It is easy to check $X(\e)$ is complete with the norm $\|\cdot\|_{X_\ga(\I)}$.
  We define the map from $X(\e)$ to $C(\R^3\times[0,\I))$ by
  \EQQS{
    \Phi[v](x,t)=\e u^0(x,t)+L((V_\ga*v^2)v)(x,t).
  }
  Let $\e_0>0$ be such that $2^4C_0^2C_1M^2\e_0^2\le 1$.
  Then, for any $\e\in(0,\e_0]$, we see from Lemma \ref{lem_free} and Proposition \ref{prop_duhamel} that
  \EQQS{
    \|\Phi[u]\|_{X_\ga(\I)}
    &\le C_0\|(u_0,u_1)\|_{Y((3-\ga)/2)}
      +C_1\|u\|_{X_\ga(\I)}^3\le 2C_0M\e,\\
    \|\Phi[u]-\Phi[v]\|_{X_\ga(\I)}
    &\le C_1(\|u\|_{X_\ga(\I)}^2+\|v\|_{X_\ga(\I)}^2)
      \|u-v\|_{X_\ga(\I)}\\
    &\le\frac{1}{2}\|u-v\|_{X_\ga(\I)}
  }
  for $u,v\in X(\e)$.
  So, the map $\Phi$ is a contraction on $X(\e)$, so we obtain the unique solution in $X(\e)$ for $\e\in(0,\e_0]$.
\end{proof}

\begin{proof}[Proof of Theorem \ref{thm_lb}]
  Following the argument of Theorem 1.3 in \cite{KS},
  we show that for sufficiently small $\e>0$ if $T>0$ satisfies $T\le \exp(C\e^{-2})$, then we can construct the local solution $u\in X_2(T)$ to \eqref{IE_u} with $(\e u_0,\e u_1)$.

  Let $(u_0,u_1)\in Y(3/2)$ and set $M:=\|(u_0,u_1)\|_{Y(3/2)}$.
  Let $C_0,C_1$ be defined by Lemma \ref{lem_free}, Proposition \ref{prop_duhamel}, respectively.
  Let $X(\e)$ be a subspace of $X_2(T)$ defined by
  \EQQS{
    X(\e):=\{u\in C(\R^3\times[0,T));\|u\|_{X_2(T)}\le 2C_0M\e\},
  }
  where $\e>0$ and $T>0$ will be chosen later.
  We claim that if
  \EQS{\label{eq3.1}
    2^5C_1 C_2^2 M^2\e^2 \log(3+T)\le 1,
  }
  then the following sequence $\{U_n\}_{n\in\N}\subset X(\e)$ is Cauchy in $X(\e)$:
  $$U_1=\e u^0,\quad U_{n+1}=\e u^0+L((V_2*U_n^2)U_n),\quad n\ge1.$$
  Indeed, similarly to the proof of Theorem \ref{thm_ge}, Lemma \ref{lem_free} and Proposition \ref{prop_duhamel} show $\|U_{n+1}\|_{X_2(T)}\le \e C_0M+C_1 D(T)\|U_n\|_{X_2(T)}^3$.
  Then, we can conclude that $U_n\in X(\e)$ for any $n\in\N$ by the induction on $n$ since $1\le \log(3+T)$.
  We also have for $n,m\in\N$ satisfying $n>m$
  \EQQS{
    \|U_{n+1}-U_{m+1}\|_{X_2(T)}
    &\le k \|U_{n}-U_{m}\|_{X_2(T)}
    \le k^{m}\|U_{n-m+1}-U_1\|_{X_2(T)}\\
    &\to0
  }
  as $n,m\to\I$, where $k=2^4C_1 C_2^2 M^2\e^2 \log(3+T)<1$.
  Take $\e_0>0$ so that $2^5C_1 C_2^2 M^2\e_0^2 \log6\le 1$.
  For $\e\in(0,\e_0]$, let $T>0$ be such that $2^5C_1 C_2^2 M^2\e_0^2 \log(2T)\le 1$.
  Then \eqref{eq3.1} holds, i.e.,
  \EQQS{
    &2^5C_1 C_2^2 M^2\e^2 \log(3+T)\\
    &\le 2^5C_1 C_2^2 M^2\e^2\log\Bigg(3+\frac{1}{2}\exp(2^{-5}C_1^{-1}C_2^{-2}M^{-2}\e^{-2})\Bigg)\le 1.
  }
  Therefore, we obtain $T(\e)\ge \exp(C\e^{-2})$,
  which completes the proof.
\end{proof}

\section{Upper bound of the lifespan}

In this section, we prove Theorem \ref{thm_bu}. The proof is shown by the iteration argument
by \cite{J79}.
First of all, we state the positivity of solutions for (\ref{IE_u}) under the condition
(\ref{blow-up-asm}).
\begin{lem}
\label{lem:positive}
Suppose that the assumptions in Theorem \ref{thm_bu} are fulfilled.
Let $T>0$ and let $u\in C(\R^3\times[0,T))$ be the solution of (\ref{IE_u}).
Then we have $u(x,t)>0$ for $(x,t)\in \R^3\times[0,T)$.
\end{lem}
The proof easily follows from comparison argument by Keller \cite{K57}.
We shall omit the proof.

Next, we derive a lower bound of the solution to (\ref{IE_u}) by using (\ref{blow-up-asm}).
For $l\ge1$ and $T>0$, we define
\[
\Sigma(l):=\{(|x|,t)\in [0,\infty)\times [0,T):\ t-|x|\ge l \}
\]
and
\[
\wt{\Sigma}(l):=\{(x,t)\in\R^3\times[0,T):\ (|x|,t)\in \Sigma(l)\}.
\]
\begin{lem}
\label{lem:first-est}
Suppose that the assumptions in Theorem \ref{thm_bu} are fulfilled. Let $T>0$ and let
$u\in C(\R^3\times[0,T))$ be the solution of (\ref{IE_u}). Then, $u$ satisfies
\begin{equation}
\label{first-est}
u(x,t)\ge \frac{C_0\e }{(t+r)(t-r)^{1/2}} \quad \mbox{in}\ \wt{\Sigma}(1),
\end{equation}
where $r=|x|$ and $C_0=B/2^{5/2}$.
\end{lem}
\begin{proof}
From Lemma \ref{lem:positive}, (\ref{blow-up-asm}) and (\ref{IE_u}), we have
\[
u(x,t)\ge u^0(x,t)=\frac{\e t}{4\pi}\int_{|\omega|=1}u_1(x+t\omega)dS_{\omega}\ge
\frac{B\e t}{4\pi}\int_{|\omega|=1}\frac{dS_{\omega}}{\LR{|x+t\omega|}^{5/2}}
\]
in $\wt{\Sigma}(1)$.
Making use of (\ref{Planewave}) and Lemma \ref{lem:KO} with $\kappa=1/2$, we obtain
\[
u(x,t)\ge \frac{B\e}{2r}\int_{t-r}^{t+r}\frac{\lambda}{\LR{\lambda}^{5/2}} d\lambda\ge
\frac{B\e}{2^{7/2}r}\int_{t-r}^{t+r}\frac{d\lambda}{\lambda^{3/2}}\ge \frac{B\e}{2^{5/2}(t+r)(t-r)^{1/2}}
\]
in $\wt{\Sigma}(1)$. This completes proof.
\end{proof}

\subsection{Iteration argument}

Our iteration argument is done by using the following estimates.
\begin{prop}
\label{prop:iteration_frame}
Suppose that the assumptions in Theorem \ref{thm_bu} are fulfilled.
Let $j\in\N$, $T>0$ and let $u\in C(\R^3\times[0,T))$ be the solution of (\ref{IE_u}).
Then, $u$ satisfies
\begin{equation}
\label{it_1}
u(x,t)\ge \frac{C_{j}}{(t+r)(t-r)^{1/2}}\left\{\log \left(\frac{t-r}{l_j}\right)\right\}^{a_j}\quad \mbox{in} \
\wt{\Sigma}(l_j).
\end{equation}
Here, $r=|x|$,
\begin{equation}
\left\{
\begin{array}{llll}
\label{C_j}
C_j=\exp\{3^{j-1}(\log (C_1(24)^{-S_j}E^{1/2}))-\log E^{1/2}\}\quad (j\ge2),\\
\d C_1=\frac{C_0^3\pi\e^3}{2\cdot3^5},
\end{array}\right.
\end{equation}
where
\begin{equation}
\label{def:S_j,E}
S_j=\sum_{k=1}^{j-1}\frac{k}{3^k}\quad \mbox{and}\quad E=\frac{\pi}{2^6\cdot3^3}.
\end{equation}
Also, $a_j$ and $l_j$ are defined by
\begin{equation}
\label{a_j}
a_j=\frac{3^j-1}{2}\quad (j\in\N),
\end{equation}
\begin{equation}
\label{l_j}
l_j=\sum_{k=0}^{j}2^{-k}\quad (j\in\N).
\end{equation}
\end{prop}

\begin{proof}
We apply the slicing method developed by Agemi, Kurokawa and Takamura \cite{AKT00}.
The proof of (\ref{it_1}) follows from the induction.

We first show that (\ref{it_1}) holds for $j=1$.
For the convolution term, the following estimate is derived by using the estimate (\ref{first-est}):
\begin{equation}
\label{conv-first1}
(V_{2}*u^2)(x,t)\ge\frac{C_0^2\pi \e^2r \log(t-r) }{(t+r)^3}\quad \mbox{in} \
\wt{\Sigma}(1).
\end{equation}
Actually, from the definition of $V_{2}$ and using the polar coordinate, we get
\begin{equation}
\label{frame-conv}
\d(V_{2}*u^2)(x,t)=\int_{\R^3}\frac{u^2(z,t)}{|x-z|^2}dz
\d=\int_{0}^{\infty}u^2(\rho \omega,t)\rho^2
\int_{|\omega|=1}\frac{dS_{\omega}}{|x-\rho\omega|^2}d\rho
\end{equation}
for $(x,t)\in \R^3\times[0,T)$.

Making use of (\ref{first-est}) and (\ref{Planewave}), we obtain
\[
\begin{array}{lll}
\d(V_{2}*u^2)(x,t)&\ge
\d C_0^2\e^2\int_{r}^{t-1}\frac{\rho^2}{(t+\rho)^2(t-\rho)}
\int_{|\omega|=1}\frac{dS_{\omega}}{|x-\rho\omega|^2}d\rho\\
&\d=\frac{2 C_0^2\pi\e^2}{r}\int_{r}^{t-1}\frac{\rho}{(t+\rho)^2(t-\rho)}
\int_{\rho-r}^{\rho+r}\frac{d\eta}{\eta}d\rho\\
&\d \ge 2C_0^2\pi \e^2 \int_{r}^{t-1}\frac{1}{(t+\rho)^2(t-\rho)}
\int_{\rho-r}^{\rho+r}\frac{d\eta}{\eta}d \rho
\end{array}
\]
in $\wt{\Sigma}(1)$.
Noticing that
\begin{equation}
\label{conv-est2}
\int_{\rho-r}^{\rho+r}\frac{d\eta}{\eta}\ge\frac{2r}{\rho+r}\ge\frac{2r}{t+r}
\end{equation}
for $t-r\ge1$ and $t\ge \rho$, we have
\[
\begin{array}{lll}
\d(V_{2}*u^2)(x,t)
&\d\ge
\frac{4C_0^2\pi \e^2 r}{t+r}\int_{r}^{t-1}\frac{d \rho}{(t+\rho)^2(t-\rho)} \\
&\d\ge \frac{C_0^2\pi \e^2r }{(t+r)^3}\int_{r}^{t-1}\frac{d\rho}{t-\rho}
\d=\frac{C_0^2\pi\e^2 r \log(t-r) }{(t+r)^3}
\end{array}
\]
in $\wt{\Sigma}(1)$. Thus, we see that (\ref{conv-first1}) is true.

We next estimate for the Duhamel term by using the estimates (\ref{conv-first1}) and (\ref{first-est}).
Let $\chi_{A}$ be a characteristic function on a set $A$. Here, we do not distinguish between
$\chi_{\Sigma(l)}$ and $\chi_{\wt{\Sigma}(l)}$ for $l\ge1$.

By (\ref{conv-first1}) and (\ref{first-est}), we note that
\begin{equation}
\label{conv-first1-1}
\chi_{\wt{\Sigma}(1)}(x,t)(V_{2}*u^2)(x,t)\ge \chi_{\Sigma(1)}(r,t)
\frac{C_0^2\pi \e^2r \log(t-r) }{(t+r)^3}
\end{equation}
and
\begin{equation}
\label{first-est-1}
\chi_{\wt{\Sigma}(1)}(x,t)u(x,t)\ge \chi_{{\Sigma(1)}}(r,t)\frac{C_0\e }{(t+r)(t-r)^{1/2}}
\end{equation}
hold for $(x,t)\in\R^3\times[0,T)$.

Let $(x,t)\in \wt{\Sigma}(l_1)(=\wt{\Sigma}(3/2))$. Noting the positivity of the linear term of (\ref{IE_u}) and making use of the estimates
(\ref{conv-first1-1}), (\ref{first-est-1}) and Lemma \ref{lem:est-du}, we have
\[
\begin{array}{lll}
u(x,t)&\ge L((V_{2}*u^2)u)(x,t)\\
&\ge L(\chi_{\wt{\Sigma}(1)}(V_{2}*u^2)u)(x,t) \\
&\d\ge\frac{C_0^3\pi\e^3}{2r} \iint_{D(r,t)\cap \Sigma(1)}
\frac{\lambda^2\log(s-\lambda) }{(s+\lambda)^4(s-\lambda)^{1/2}}d\lambda ds
\end{array}
\]
in $\wt{\Sigma}(l_1)$.
Changing the variables in the above integral by
\begin{equation}
\label{alpha_beta}
\alpha=s+\lambda,\  \beta=s-\lambda,
\end{equation}
we get
\[
\begin{array}{lll}
u(x,t)&\d\ge \frac{C_0^3\pi\e^3}{4r}\int_{1}^{t-r}\frac{\log \beta }{\beta^{1/2}}
\int_{t-r}^{t+r}\frac{(\alpha-\beta)^2}{\alpha^4}d\alpha d\beta\\
&\d\ge \frac{C_0^3\pi\e^3}{4r(t-r)^{1/2}}\int_{1}^{t-r}(t-r-\beta)^2\log \beta
\int_{t-r}^{t+r}\frac{d\alpha}{\alpha^4}d\beta
\end{array}
\]
in $\wt{\Sigma}(l_1)$. Applying Lemma \ref{lem:KO} with $\kappa=3$ to the $\alpha$-integral,
we obtain
\[
u(x,t)\d\ge \frac{C_0^3\pi\e^3}{6(t+r)(t-r)^{7/2}}
\int_{1}^{t-r}(t-r-\beta)^2\log \beta d\beta
\]
in $\wt{\Sigma}(l_1)$.
Noticing that $2(t-r)/3\ge1$ for $(r,t)\in \Sigma(l_1)$, we obtain
\[
\begin{array}{llll}
u(x,t)&\d\ge \frac{C_0^3\pi\e^3}{6(t+r)(t-r)^{7/2}}
\int_{2(t-r)/3}^{t-r}(t-r-\beta)^2\log \beta d\beta\\
&\d \ge \frac{C_0^3\pi\e^3\log (2(t-r)/3)}{6(t+r)(t-r)^{7/2}}
\int_{2(t-r)/3}^{t-r}(t-r-\beta)^2d\beta\\
&\d\ge \frac{C_0^3\pi\e^3}{2\cdot3^5(t+r)(t-r)^{1/2}}\log \{2(t-r)/3\}
\end{array}
\]
in $\wt{\Sigma}(l_1)$. Hence the estimate (\ref{it_1}) holds for $j=1$.

Assume that (\ref{it_1}) holds for $j\in\N$. From Lemma 3.1 in \cite{W17}, we note that the sequence $C_j$ in (\ref{C_j}) satisfies the following relation:
\begin{equation}
\label{C_j-new}
C_{j+1}=\frac{C_j^3E}{24^j} \quad (j\in\N),
\end{equation}
where $E$ is defined in (\ref{C_j}).

Similarly to the proof of (\ref{conv-first1}),
we first derive the estimate for the convolution term by using (\ref{it_1}):
\begin{equation}
\label{conv-est3}
(V_{2}*u^2)(x,t)\ge\frac{C_j^2\pi r}{(2a_j+1)(t+r)^3} \left\{\log\left((t-r)/l_j\right)\right\}^{2a_j+1}
\quad \mbox{in} \
\wt{\Sigma}(l_{j}).
\end{equation}

Indeed, putting the estimate (\ref{it_1}) to the integral of (\ref{frame-conv}) and using (\ref{Planewave}) and
(\ref{conv-est2}),
we get
\[
\begin{array}{lll}
(V_{2}*u^2)(x,t)
&\d\ge \frac{2C_j^2\pi}{r}\int_{r}^{t-l_j}\frac{\rho\left\{\log \left((t-\rho)/l_j\right)\right\}^{2a_j}}
{(t+\rho)^2(t-\rho)}\int_{\rho-r}^{\rho+r}\frac{d\eta}{\eta}d\rho\\
&\ge\d \frac{4C_j^2\pi r}{t+r}\int_{r}^{t-l_j}\frac{\left\{\log \left((t-\rho)/l_j\right)\right\}^{2a_j}}
{(t+\rho)^2(t-\rho)}d\rho\\
&\ge \d \frac{C_j^2\pi r}{(t+r)^3}\int_{r}^{t-l_j}\frac{\left\{\log \left((t-\rho)/l_j\right)\right\}^{2a_j}}
{t-\rho}d\rho\\
&=\d \frac{C_j^2\pi r}{(2a_j+1)(t+r)^3}\int_{r}^{t-l_j}-\frac{\p }{\p \rho}
\left\{\log \left((t-\rho)/l_j\right)\right\}^{2a_j+1}d\rho\\
&=\d \frac{C_j^2\pi r}{(2a_j+1)(t+r)^3} \left\{\log\left((t-r)/l_j\right)\right\}^{2a_j+1}
\end{array}
\]
in $\wt{\Sigma}(l_{j})$. Thus, we see that (\ref{conv-est3}) is true.

We next estimate for the Duhamel term by using the estimates (\ref{conv-est3}) and (\ref{it_1}).
By (\ref{conv-est3}) and (\ref{it_1}), we note that
\begin{equation}
\label{conv-first3-1}
\chi_{\wt{\Sigma}(l_j)}(x,t)(V_{2}*u^2)(x,t)\ge\chi_{\Sigma(l_j)}(r,t)\frac{C_j^2\pi r}{(2a_j+1)(t+r)^3} \left\{\log\left((t-r)/l_j\right)\right\}^{2a_j+1}
\end{equation}
and
\begin{equation}
\label{it_1-1}
\chi_{\wt{\Sigma}(l_j)}(x,t)u(x,t)\ge \chi_{\Sigma(l_j)}(r,t)\frac{C_{j}}{(t+r)(t-r)^{1/2}}
\left\{\log\left((t-r)/l_j\right)\right\}^{a_j}
\end{equation}
hold for $(x,t)\in\R^3\times[0,T)$.

Let $(x,t)\in \wt{\Sigma}(l_{j+1})$. Making use of the the positivity of the linear term of (\ref{IE_u}), and (\ref{conv-first3-1}), (\ref{it_1-1}) and Lemma \ref{lem:est-du}, we have
\[
\begin{array}{lll}
u(x,t)
&\d\ge L(\chi_{\wt{\Sigma}(l_j)}(V_{2}*u^2)u)(x,t) \\
&\d\ge\frac{C_j^3\pi}{2(2a_j+1)r} \iint_{D(r,t)\cap\Sigma(l_j)}
\frac{\lambda^2\left\{ \log ((s-\lambda)/l_j)\right\}^{3a_j+1}}{(s+\lambda)^4(s-\lambda)^{1/2}}d\lambda ds
\end{array}
\]
in $\wt{\Sigma}(l_{j+1})$.
Changing the variables in the above integral by (\ref{alpha_beta}), we get
\[
\begin{array}{lll}
u(x,t)&\d\ge \frac{C_j^3\pi}{4(2a_j+1)r}\int_{l_j}^{t-r}\frac{\left\{\log (\beta/l_j)\right\}^{3a_j+1} }{\beta^{1/2}}
\int_{t-r}^{t+r}\frac{(\alpha-\beta)^2}{\alpha^4}d\alpha d\beta\\
&\d\ge \frac{C_j^3\pi}{4(2a_j+1)r(t-r)^{1/2}}\int_{l_j}^{t-r}(t-r-\beta)^2
\left\{\log (\beta/l_j)\right\}^{3a_j+1}
\int_{t-r}^{t+r}\frac{d\alpha}{\alpha^4}d\beta
\end{array}
\]
in $\wt{\Sigma}(l_{j+1})$.  It follows from Lemma \ref{lem:KO} with $\kappa=3$ that
\[
u(x,t)\d\ge \frac{C_j^3\pi}{6(2a_j+1)(t+r)(t-r)^{7/2}}
\int_{l_j}^{t-r}(t-r-\beta)^2\left\{\log (\beta/l_j)\right\}^{3a_j+1}d\beta
\]
in $\wt{\Sigma}(l_{j+1})$.
We note that $l_j(t-r)/l_{j+1}\ge l_j$ holds for $(r,t)\in \Sigma(l_{j+1})$.
Similarly to the proof of the case $j=1$, we obtain
\[
\begin{array}{llll}
u(x,t)&\d\ge \frac{C_j^3\pi}{6(2a_j+1)(t+r)(t-r)^{7/2}}
\int_{l_j(t-r)/l_{j+1}}^{t-r}(t-r-\beta)^2\left\{\log (\beta/l_j)\right\}^{3a_j+1} d\beta\\
&\d \ge \frac{C_j^3\pi\left\{\log((t-r)/l_{j+1})\right\}^{3a_j+1}}{6(2a_j+1)(t+r)(t-r)^{7/2}}
\int_{l_j(t-r)/l_{j+1}}^{t-r}(t-r-\beta)^2d\beta\\
&\d\ge \frac{C_j^3\pi(1-l_j/l_{j+1})^3}{2\cdot3^2(3a_j+1)(t+r)(t-r)^{1/2}}\left\{\log ((t-r)/l_{j+1})\right\}^{3a_j+1}
\end{array}
\]
in $\wt{\Sigma}(l_{j+1})$. Since $1<l_j<2$, we have $1-l_j/l_{j+1}=2^{-(j+1)}/l_{j+1}\ge 2^{-(j+2)}$.
Recalling the definition of $a_j$, we get $a_{j+1}=3a_j+1\le 3^{j+1}/2$.
From (\ref{C_j-new}), we obtain
\[
\begin{array}{lll}
u(x,t)&\d\ge \frac{C_j^3\pi}{2^6\cdot3^3\cdot24^{j}(t+r)(t-r)^{1/2}}
\left\{\log ((t-r)/l_{j+1})\right\}^{3a_j+1}\\
&\d=C_{j+1}\frac{\left\{\log ((t-r)/l_{j+1})\right\}^{3a_j+1}}{(t+r)(t-r)^{1/2}}
\end{array}
\]
in $\wt{\Sigma}(l_{j+1})$.  Therefore, (\ref{it_1}) holds for all $j\in\N$.
The proof of Proposition \ref{prop:iteration_frame} is now completed.
\end{proof}
\vskip10pt
\par\noindent
\subsection{Proof of Theorem \ref{thm_bu}}

Theorem \ref{thm_bu} is proved by contradiction argument.
\begin{proof}[Proof of Theorem \ref{thm_bu}.]
Taking $\e_0=\e_0(u_1)>0$ so small that
\[
\exp(F^{-2/3}\e_0^{-2})>4,
\]
where we set
\begin{equation}
\label{const:F}
F=C_0^3\pi2^{-1}\cdot3^{-5}(24)^{-S}E^{1/2}>0.
\end{equation}
Here, $C_0$ $S$ and $E$ are defined in (\ref{first-est}), (\ref{def:S}) and (\ref{def:S_j,E}) respectively.
Next, for a fixed $\e\in(0,\e_0]$, we suppose that $T$ satisfies
\begin{equation}
\label{T_asm_1}
T>\exp(2F^{-2/3}\e^{-2})\ (>4).
\end{equation}
Let $u\in C(\R^3\times[0,T))$ be the solution of (\ref{IE_u}) satisfying (\ref{T_asm_1}).
Setting
\begin{equation}
\label{def:S}
S=\lim_{j\rightarrow \infty}S_j\left(=\sum_{k=1}^{j-1}\frac{k}{3^k}\right),
\end{equation}
we see that $S_j\le S$ for all $j\in\N$.
Since the definitions of $C_1$ in (\ref{C_j}) and (\ref{const:F}), the sequence $C_j$ in
(\ref{C_j}) implies
\begin{equation}
\label{C_j_est}
\begin{array}{llll}
C_{j}&\ge\exp\{3^{j-1}\{\log(C_{1}(24)^{-S}E^{1/2})\}-\log E^{1/2}\}&\\
&=E^{-1/2}\exp\{3^{j-1}\{\log(C_{1}(24)^{-S}E^{1/2})\}\}&\\
&=E^{-1/2}\exp\{3^{j-1}\{\log(\e^3F)\}\}.&
\end{array}
\end{equation}
\par
Let $(x,t)\in \wt{\Sigma}(2)$. Combining (\ref{C_j_est}) with (\ref{it_1}) and noting $l_j<2$, we have
\[
\begin{array}{ll}
\d u(x,t)\d\ge E^{-1/2}\exp\{3^{j-1}\{\log(\e^3F)\}\}
\frac{\left\{\log ((t-r)/2)\right\}^{(3^j-1)/2}}{(t+r)(t-r)^{1/2}}
\end{array}
\]
in $\wt{\Sigma}(2)$.
Since
\[
\left(\log \frac{t-r}{2}\right)^{(3^j-1)/2}=
\exp\left\{3^{j-1}\left\{\log\left(\log\frac{t}{4}\right)^{3/2}\right\}\right\}
\left(\log\frac{t}{4}\right)^{-1/2}
\] in $\Gamma:=\{r=t/2\}(\subset  \Sigma(2))$,
we get
\[
u(x,t)
\ge2^{3/2}\cdot3^{-1}E^{-1/2}\exp\{3^{j-1}K(t)\}t^{-3/2}\left\{\log (t/4)\right\}^{-1/2}\\
\]
in $\wt{\Gamma}:=\{(x,t)\in\R^3\times[0,T): (r,t)\in \Gamma\}$, where we set
\[
K(t)=\log \left\{\e^3F\left\{\log (t/4)\right\}^{3/2}\right\}.
\]
\par
By (\ref{T_asm_1}) and the definition of $F$,
we have $K(T)>0$.
Therefore, we get $u(x,t)\rightarrow \infty$ as
$j\rightarrow \infty$ in $\wt{\Gamma}$. The proof of Theorem \ref{thm_bu} is now completed.
\end{proof}

\section*{Acknowledgement}
The first author was supported by Grant-in-Aid for JSPS Research Fellow 20J12750.
The second author has been supported
by the Grant-in-Aid for Scientific Research (B) (No.18H01132), the Grant-in-Aid for Scientific Research (B) (No.19H01795) and Young Scientists Research (No. 20K14351), Japan Society for the Promotion of Science.


\bibliographystyle{plain}

\end{document}